\documentclass[11pt, reqno]{amsart}
\usepackage{euscript,amscd,amsgen,amsfonts,amssymb,latexsym}

\newcommand{\eqdef}{\stackrel{\scriptscriptstyle\rm def}{=}}
\usepackage{bbm}
\usepackage{pgfplots, multicol}
\usepackage{mathtools}
\usepackage{paralist}
\usepackage{todonotes}
\usepackage{enumerate}
\usepackage{array}
\usepackage{graphicx}

\newtheorem{theorem}{Theorem}[section]

\newtheorem{proposition}[theorem]{Proposition}
\newtheorem{corollary}[theorem]{Corollary}
\newtheorem{definition}[theorem]{Definition}
\newtheorem{lemma}[theorem]{Lemma}

\newtheorem{example}[theorem]{Example}

\newcommand{\beha}{\begin{enumerate}}
\newcommand{\behe}{\end{enumerate}}
\renewcommand{\epsilon}{\varepsilon}

\renewcommand{\refeq}[1]{(\ref{#1})}

\newcommand{\Per}{{\rm Per}}
\newcommand{\EPer}{{\rm EPer}}

\newcommand{\Or}{\mathcal{O}}

\DeclareMathOperator{\GS}{GS}

\newcommand{\cM}{\EuScript{M}}

\newcommand{\cH}{\EuScript{H}}
\newcommand{\bR}{{\mathbb R}}

\newcommand{\bZ}{{\mathbb Z}}
\newcommand{\bN}{{\mathbb N}}
\newcommand{\bQ}{{\mathbb Q}}

\newcommand{\cA}{{\mathcal A}}

\newcommand{\cC}{{\mathcal C}}
\newcommand{\cU}{{\mathcal U}}
\newcommand{\cV}{{\mathcal V}}

\newcommand{\cS}{{\mathcal S}}

\def\1{1\!\!1}

\def\and{\text{ and }}

        \def\conv{\text{{\rm conv}}}

\def\h{{\text h}}
           \def\htop{{\text h_{{\rm top}}}}
 \def\Ptop{{\rm P_{\rm top}}}

          \def\cO{\mathcal O}

\def\ES{{\rm ES}}

                        \def\^{\widetilde}

\def\Per{{\rm Per}}
\def\inn{{\rm int}}

\def\var{{\rm var}}

\def\Per{{\rm Per}}
\def\EPer{{\rm EPer}}

\def\1{1\!\!1}

\def\inn{{\rm int\ }}

\newtheorem*{thmA}{Theorem A}

\newtheorem*{thmB}{Theorem B}

\newtheorem*{thmC}{Theorem C}

\DeclareMathSymbol{\varnothing}{\mathord}{AMSb}{"3F}

\title{Computability at zero temperature}

\author{Michael Burr}
\address{Department of Mathematical Sciences, Clemson University, Clemson SC, 29634, USA}\email{burr2@clemson.edu}
\thanks{Burr was partially supported by the National Science Foundation Grant CCF-1527193. }
\thanks{Wolf partially supported by a grant from the PSC-CUNY (TRADB-49-253 to Christian Wolf)}

\author{Christian Wolf}\address{Department of Mathematics, The City College of New York, New York, NY, 10031, USA}\email{cwolf@ccny.cuny.edu}

\begin{document}

\begin{abstract}
In this paper, we investigate  the computability of thermodynamic invariants at zero temperature for one-dimensional subshifts of finite type. In particular, we prove that the residual entropy (i.e., the joint ground state entropy) is an upper semi-computable  function on the space of continuous potentials, but it is not computable.  Next, we consider locally constant potentials for which the zero-temperature measure is known to exist. We  characterize the computability of the zero-temperature measure and its entropy for potentials that are constant on cylinders of a given  length $k$.  In particular, we show the  existence of an open and dense set of locally constant potentials for which the zero-temperature measure can be computationally identified as an elementary periodic point measure. Finally, we show that our methods do not generalize to treat the case when $k$ is not given. 
\end{abstract}
\keywords{Zero-temperature measures, residual entropy, ground states, entropy, thermodynamic formalism, computability}
\subjclass[2010]{Primary 37D35, 37E45, 03D15 Secondary 37B10,  37L40, 03D80}
\maketitle

\section{Introduction}

\subsection{Motivation}
It is a natural and important question to understand which mathematical invariants can (in principle) be derived by computer experiments.  In particular, since computer-based approximations are often used to gain insight into theoretical questions, estimates on the quality and accuracy of computational results may be needed to have confidence in any conjectures drawn from such experiments.  The answers to these questions (and the corresponding estimates) are naturally linked to questions about mathematical proofs and models.  In fact, these answers lie at the boundary of mathematics, computer science, and their applications.

In this paper, we provide some answers concerning the computability of basic thermodynamic invariants at zero temperature. 
In particular, we study the computability of the residual entropy (which coincides with the entropy of the ground states of the system) on the space of continuous potentials for subshifts of finite type (SFTs). We show that the residual entropy is an upper semi-computable function of the potentials, but it is not computable.  One complication that arrises is that continuous potentials may have phase transitions, which do not occur in the  H\"older continuous case. Since, in general, phase transitions cannot be detected algorithmically, see, e.g., \cite{Sp,Yap:Zero}, we are required to develop a new approach which is based on techniques from convex analysis and the thermodynamic formalism.

We also consider the computability of the zero-temperature measure for locally constant potentials. The existence of this measure was originally established by Br\'emont \cite{Br} by using methods from analytic geometry (for existence proofs using methods from dynamical systems, see \cite{CGU,L}).  For potentials that are constant on cylinders of a given length $k$, we provide explicit characterizations of the sets of potentials for which the zero-temperature measure or its entropy are computable.  We explicitly describe an open and dense subset $\cO_k$ of computable, locally constant potentials for which the zero-temperature measure is a computable periodic point measure.
As a counterpart to these results, we show that once we consider the space of all locally constant potentials (i.e., without fixing the cylinder length $k$), then the set $\cO=\bigcup_k \cO_k$ has empty interior. In particular, this shows that our results do not directly generalize to the case when $k$ is not given.

In the literature, there are several recent papers that study invariant sets, topological entropy, and other invariants from the computable analysis point of view.  These papers include results about the computability of certain specific measures (e.g., maximal entropy and physical measures), see \cite{BBRY,GHR} and the references therein.  Furthermore, there are papers proving results on the numerical computation of invariant sets, entropy, and dimension, see, e.g., \cite{Co,JP1,JP2} and the references therein.  There are also studies concerning the computation of the topological entropy or pressure for one and multi-dimensional shift maps, see, e.g., \cite{ HS,HM,Pa, PS1,Sc,Sp}.  In our recent paper with Schmoll \cite{BSW}, we derive results about the computability of generalized rotation sets and localized entropies.  In particular, our results hold for SFTs.  We note that the results in \cite{BSW} only consider the case of positive temperature, while the 
more delicate case of zero temperature is considered in this paper.  To the best of our knowledge, this paper is the first attempt to study the  computability of thermodynamic invariants at zero-temperature.

\subsection{Statement of results.}\label{sec:results}  Let $f:X\to X$ be a  subshift of finite type (SFT) over an alphabet with $d$ elements, and  let $\cM$ be the set of $f$-invariant Borel probability measures on $X$ endowed with the weak$^\ast$ topology.  This topology makes  $\cM$ a  compact, convex, and metrizable topological space.  
In this paper, we use as the standing assumption that $f$ is transitive and has positive topological entropy.
We consider
 the Banach space $(C(X, \bR),\|\cdot\|_\infty)$,  where $\|\cdot\|_\infty$ denotes the supremum norm.
For  $\phi\in C(X,\bR)$ and $\mu\in\cM$, we write
$\mu(\phi)=\int \phi\, d\mu$ and define
$$
I(\phi)=\left\{\mu(\phi): \mu \in \cM\right\}.
$$
It follows, from the compactness and convexity of $\cM$, that $I(\phi)$ is a compact interval $[a_\phi,b_\phi]$. We define $\cM_{\max}(\phi)=\{\mu\in \cM: \mu(\phi)=b_\phi\}.$
If $\mu\in \cM_{\max}(\phi)$, then we say $\mu$ is a {\em maximizing measure} for $\phi$. Moreover, we say $\phi\in C(X,\bR)$ is {\em uniquely maximizing} if $\cM_{\max}(\phi)$ is a singleton. 
We note that the study of maximizing measures is one of the central objectives in the area of ergodic optimization. We refer the reader to the survey article \cite{Je2} for a state-of-the-art presentation. 
We call 
$$
h_{\infty,\phi} =\sup\{h_\mu(f): \mu\in \cM_{\max}(\phi)\}
$$
 the {\em residual entropy} of the potential $\phi$. The residual entropy coincides with the entropy of the ground states of the potential $\phi$ (see  Section   \ref{sec: ground states} for details). In particular, if the zero-temperature measure $\mu_{\infty,\phi}$ of $\phi$ exists (see below and Section 2.2 for the definition of zero-temperature measures), then $h_{\infty,\phi}$ coincides with the entropy of $\mu_{\infty,\phi}$.
 
There are several recent theoretical results about the residual entropy and uniquely maximizing periodic point measures for an open and dense set of potentials in the H\"older and Lipschitz topologies \cite{C,CLT,M,QS}.  We observe, however, that these topologies are not compatible with the supremum topology since open balls in the supremum topology are not bounded in the H\"older and Lipschitz topologies.  Therefore, it does not appear possible to study these genericity results from the computable analysis point of view, see Section \ref{sec:compute:basic} for details.  Consequently, the work in this paper uses the supremum norm.

Our first goal is to characterize the computability of the function $\phi\mapsto h_{\infty,\phi}$.  In this paper, we use two notions of computability for functions: Computable functions and upper semi-computable functions (also called right recursively enumerable or right computable functions).  We say that a function $h:C(X,\bR)\to \bR$ is computable if, for any input function $\phi$, the real number $h(\phi)$ can be calculated to any prescribed accuracy.  Upper semi-computability is a weaker notion of computability, where, instead, there is an algorithm to compute a sequence $q_n$ converging to $h(\phi)$ from above.  In particular, for upper semi-computability, the bounds on the convergence rate for $q_n\rightarrow h(\phi)$ are not included. We refer the reader to Section 2.5 and \cite{BY,GHR} for details.

We note that solely using computability, equality is not decidable.  More precisely, we observe that when $h$ is computable, we can only calculate $h(\phi)$ up to some error.  Therefore, we may conclude that $h(\phi)$ is in a small interval, but we cannot conclude which point in the interval equals $h(\phi)$.  For more details, see Section \ref{sec:compute:basic} and \cite{Sp,Yap:Zero}.

The first main theorem we prove in this paper shows that the residual entropy is semi-computable, but not  computable.

\begin{thmA}
The function $\phi\mapsto h_{\infty,\phi}$ is upper semi-computable, but not computable on $C(X,\bR)$.  Moreover, the map $\phi\mapsto\h_{\infty,\phi}$ is continuous at $\phi_0$ if and only if $h_{\infty,\phi_0}=0$.
\end{thmA}

In Section \ref{sec:compute:basic}, we introduce the definition for the function $\phi\mapsto h_{\infty,\phi}$ to be computable at a point $\phi_0$.  This definition provides a computable version of being continuous at a point.  Moreover, if the map $\phi\mapsto h_{\infty,\phi}$ restricted to a set $S$ is computable and $\phi_0$ is in the interior of $S$, then $\phi\mapsto h_{\infty,\phi}$ is computable at $\phi_0$.  With this definition in hand, a direct corollary of Theorem A is: 

\begin{corollary}\label{cor:computabilityentropy}
The function $\phi\mapsto h_{\infty,\phi}$ is computable at $\phi_0$ if and only if $h_{\infty,\phi_0}=0$.
\end{corollary}

The second goal of this paper is to study the computability of the zero-temperature measure and its entropy for locally constant potentials.
We recall that $\mu\in \cM$ is an equilibrium state of $\phi\in C(X,\bR)$ if $\mu$ maximizes $h_\nu(f)+\nu(\phi)$ among all $\nu\in \cM$. If $\phi$ is H\"older continuous  (and, in particular, if $\phi$ is locally constant), then the equilibrium state is unique and we denote it by $\mu_\phi$. We say $\mu_{\infty,\phi}$ is the  zero-temperature measure of $\phi$ if $\mu_{\infty,\phi}=\lim_{\beta\to\infty}\mu_{\beta\phi}$, where the limit is taken in the weak$^\ast$ topology\footnote{We point out that, in the mathematical theory of the thermodynamic formalism,
it is customary to consider the inverse temperature $\beta=1/T$ (with $T$ being the temperature of the system) and to take the limit $\beta\to\infty$. We mention that the notation that is used for the inverse temperature in physics is $\beta=\frac{1}{k_BT}$, where $k_{B}$ is Boltzmann's constant, which can be taken equal to one in an appropriate system of units.}. We recall that, for locally constant potentials, the zero-temperature measure exists \cite{Br}.  Let  $LC(X,\bR)=\bigcup_{k\in \bN} LC_k(X,\bR)$ denote the space of locally constant potentials, where $LC_k(X,\mathbb{R})$ denotes the space of potentials that are constant on cylinders of length $k$.  Let $m_c(k)$ denote the cardinality of the set of  cylinders of $X$ of length $k$ (note that $m_c(k)\leq d^k$).  Then, we can identify $LC_k(X,\mathbb{R})$ with $\mathbb{R}^{m_c(k)}$, which makes $LC_k(X,\mathbb{R})$  a Banach space, when endowed with the standard norm.

We note that, for the purpose of studying  zero-temperature measures and their associated entropies, it suffices to  consider the space   $LC_k(X,\mathbb{R})\cap \overline{B}(0,1)$, where $\overline{B}(0,1)$ is the closed unit ball in $\mathbb{R}^{m_c(k)}$. This reduction follows since $\mu_{\infty,\phi}=\mu_{\infty,\alpha \phi}$ for all $\alpha>0$.  
To illustrate some of the difficulties when dealing with the  computability of $h_{\infty,\phi}$ and $\mu_{\infty,\phi}$, we consider the following basic example, see \cite{Br,WY}:
\begin{example}
Let $X$ be the full shift on two symbols, i.e., $X=\{0,1\}^\bN$, and let $f:X\to X$ be the shift map.  Let $0<\alpha_1,\alpha_2$ be computable real numbers with $\alpha_1\approx\alpha_2$.  Let $\phi\in LC_2(X,\bR)$ be given by the matrix $\begin{pmatrix}
\alpha_1&\alpha_2\\
\alpha_2&0
\end{pmatrix},$ where $\phi_{i,j}$ denotes the value of $\phi$ on the cylinder $\cC_2(ij)\eqdef\{x: x_1=i, x_2=j\}$.  It follows that, if $\alpha_1\not=\alpha_2$, then $\mu_{\infty,\phi}$ is a periodic point measure, and, in particular, $h_{\infty,\phi}=0$. On the other hand, if $\alpha_1=\alpha_2$, then $\mu_{\infty,\phi}$ is the unique measure of maximal entropy (i.e., the Parry measure) of the Golden mean shift, i.e., the SFT with transition matrix $A=
\begin{pmatrix}
1&1\\
1&0
\end{pmatrix}$. Furthermore,  $h_{\mu_{\infty,\phi}}(f)=\log \frac{1+\sqrt{5}}{2}$.

As mentioned above, any function determined by a Turing machine cannot distinguish between the cases $\alpha_1=\alpha_2$ and $(\alpha_1\approx\alpha_2$, but $\alpha_1\not=\alpha_2)$, i.e., the condition $\alpha_1=\alpha_2$ is undecidable.  More precisely, any Turing machine for computing the entropy of $\phi$ can query $\alpha_1$ and $\alpha_2$ to arbitrary, but finite precision.  If $\alpha_1$ and $\alpha_2$ agree up to the queried precision, then the algorithm cannot distinguish $\alpha_1$ and $\alpha_2$, and, therefore, it cannot decide whether they are equal.  Consequently, neither $h_{\infty,\phi}$ nor $\mu_{\infty,\phi}$ are computable. 
\end{example}

To overcome these difficulties, we break the space of potentials $LC_k(X,\mathbb{R})\cap \overline{B}(0,1)$ into three sets with distinct computability properties, namely,
\begin{equation}\label{eq:decompositionk}
LC_k(X,\bR)\cap \overline{B}(0,1)=\cO_k\, \dot\cup\,  \cU_k\, \dot\cup\, \cV_k.
\end{equation}
We explicitly define the three sets and identify their properties:
\begin{enumerate}[(a)]
\item $\cO_k$ is the set of uniquely maximizing potentials $\phi\in LC_k(X,\bR)\cap \overline{B}(0,1)$. Moreover, the unique maximizing  measure of $\phi$ is a $k$-elementary periodic point measure.  Additionally, $\cO_k$ is open and dense in $LC_k(X,\bR)\cap\overline{B}(0,1)$.
\item $\cO_k\,\dot\cup\,\cU_k$ is the set of potentials $\phi\in LC_k(X,\bR)\cap \overline{B}(0,1)$ with $h_{\infty,\phi}=0$.  Therefore, $\cU_k$ is the set of potentials with more than one ergodic maximizing measure, all of which are  $k$-elementary periodic point measures.\footnote{We note that this condition implies $h_{\infty,\phi}=0$ for all $\phi \in \cU_k$, see \cite{WY}.} 
Furthermore, for $\phi\in \cU_k$, the measure $\mu_{\infty,\phi}$ is a convex combination of these $k$-elementary periodic point measures.  It follows that $\cO\,\dot\cup\,\cU_k$ is an open set in $LC_k(X,\bR)\cap\overline{B}(0,1)$.
\item $\cV_k$ is the set of potentials $\phi\in LC_k(X,\bR)\cap \overline{B}(0,1)$ with $h_{\infty,\phi}>0$.  It follows that $\cV_k$ is a closed set in $LC_k(X,\bR)\cap\overline{B}(0,1)$.
\end{enumerate}
The properties of the sets described in this partition follow from results in \cite{WY}, where a similar topological partition is considered. We note that  the statement that $\cO_k\,\dot\cup\,\cU_k$ is open is not explicitly proven in \cite{WY}, but its proof
 is analogous to the proof that $\cO_k$ is open.

To be able to make statements about the computability of the sets $\cO_k$ and $\cO_k\,\dot\cup\, \cU_k$, we briefly recall the notion of recursively open sets.  Namely, we say an open set $S$ 
is recursively open if there exists a Turing machine which produces, for each $n\in \bN$, a ball $B_n$ in $X$ such that $S=\bigcup_n B_n$, see Section \ref{sec:compute:basic} for details.

We prove the following result:

\begin{thmB}
Let $k\in \bN$ be given.  Then, the following hold:
\begin{enumerate}[(a)]
\item
The maps $\phi\mapsto\mu_{\infty,\phi}$ and $\phi\mapsto h_{\infty,\phi}$ are computable functions on $\cO_k\subset LC_k(X,\bR)$.  Furthermore, the set $\cO_k$ is a recursively open set;
 \item
The map $\phi\mapsto h_{\infty,\phi}$ is a computable function on $\cO_k\,\dot\cup\,\cU_k\subset LC_k(X,\bR)$.  For any $\phi_0\in\cU_k$, the map $\phi\mapsto \mu_{\infty,\phi}$ is not continuous (and hence not computable) at $\phi_0$ in $\cO_k\,\dot\cup\,\cU_k$.  Furthermore, the set $\cO_k\,\dot\cup\,\cU_k$ is a recursively open set; and
\item
For any $\phi_0\in\cV_k$, neither the map $\phi\mapsto h_{\infty,\phi}$ nor the map $\phi\mapsto \mu_{\infty,\phi}$ are continuous (and hence not computable) at $\phi_0$ in $LC_k(X,\bR)$.
\end{enumerate}
\end{thmB}

We point out that in the statement of Theorem B the number $k$ (i.e., the cylinder length on which the potentials are constant) is given, and, in particular, is not determined by the Turing machine that queries an oracle of the potential.  One might suspect that either $k$ can be calculated or that some of the results in Theorem $B$ generalize to $LC(X,\bR)$ without specifying $k$.  For instance, recall that $\cO=\bigcup_k \cO_k$ denotes the set of locally constant potentials that that are uniquely maximizing.  One might hope that $\cO$ is a recursively open set, i.e., that membership in $\cO$ is semi-decidable.  A first indication that this could not be true is given
Example \ref{ex4.1}  where it is shown that $\cO$ is not open in $LC(X,\bR)$. In fact, we have the following even stronger result from Proposition \ref{prop:UV}:
\begin{thmC}
 Let $f:X\to X$ be a transitive SFT with positive topological entropy. Then the set $\cO$ has no interior points in $LC(X,\bR)$.
 \end{thmC}
 As noted above, Theorem C indicates that, from the point of view of computable analysis, there are significant differences between the cases of a given and of an arbitrary cylinder length. On the other hand, Theorem C should also be of theoretical interest in ergodic optimization. This is, in part, as it displays a sharp contrast between the locally constant case and the  Lipschitz case (in the Lipschitz topology) since for the latter the set of  potentials with a uniquely maximizing periodic point measure is open and dense
 in the space of all Lipschitz potentials, see Contreras's Theorem \cite{C}.
 
 \subsection{Outline of paper}
 
The remainder of this paper is organized as follows: In Section~\ref{sec:2}, we review some
 concepts from symbolic dynamics, the thermodynamic formalism, and computational analysis. Moreover, we establish some preliminary results about the residual entropy.
In Section~\ref{sec:3}, we discuss the computability of the residual entropy as a function on the space of continuous potentials for SFTs. 
Section~\ref{sec:4} is devoted to the study of the computability of the zero-temperature measure for locally constant potentials that are constant on cylinders of a given length. 
 Finally, in Section~\ref{sec:5}, we provide a proof of Theorem C.

\section{Settings and Generalities}\label{sec:2}

In this section, we introduce the relevant background material and obtain preliminary results.  In particular, we provide overviews of the pertinent results and definitions from shift spaces, zero-temperature measures, ground states, locally constant potentials, and computability theory.

\subsection{Shift maps}

In this section, we recall the relevant material from symbolic dynamics, see, e.g., \cite{Kit} for more details.  Let $\cA=\{0,\cdots,d-1\}$ be a finite alphabet with $d$ symbols. The (one-sided) shift space $\Sigma_d$ on the alphabet $\cA$ is the set of
all sequences $x=(x_n)_{n=1}^\infty$, where $x_n\in \cA$ for all $n\in \bN$.  We endow $\Sigma_d$ with the {\em Tychonov product topology},
which makes $\Sigma_d$ a compact metrizable space. Given $0<\theta<1$, the distance function
\begin{equation}\label{defmet}
d(x,y)=d_\theta(x,y)\eqdef\theta^{\inf\{n\in \bN\,:\,x_n\not=y_n\}}\quad {\rm and}\quad d(x,x)=0
\end{equation}
defines a metric which induces this topology on $X$.
The shift map $f:\Sigma_d\to \Sigma_d$ (defined by $f(x)_n=x_{n+1}$) is a continuous $d$ to $1$ map on $\Sigma_d$.
If $X\subset \Sigma_d$ is an $f$-invariant set, then we  say  that $f|_X$ is a subshift. In the following, we use the symbol $X$ for any shift space including the full shift $X=\Sigma_d$.

Given $x\in X$, we write $\pi_k(x)=x_1\cdots x_k$ for the initial segment of length $k$ of $x$. We call $\tau=\tau_1\cdots\tau_k\in \cA^k$ a segment of length $k$ or simply a segment, when the length $k$ is understood. Moreover, we denote  the \emph{cylinder} generated by $\tau$ by $$\cC_k(\tau)=\{x\in X: x_1=\tau_1,\cdots, x_n=\tau_k\}.$$  Given $x\in X$ and $k\in \bN$, we call $\cC_k(x)=\cC_k(\pi_k(x))$ the cylinder of length $k$ generated by $x$, i.e., the cylinder consisting of all $y\in X$ that agree on the first $k$ values.   We denote the \emph{periodic point} generated by $\tau$ by $$\Or(\tau)=\tau_1 \cdots \tau_k \tau_1 \cdots \tau_k\tau_1 \cdots \tau_k\cdots,$$ provided $\Or(\tau)\in X$.  
We denote the set of all periodic points of $f$ with period $n$ by $\Per_n(f)$. Moreover, $\Per(f)=\bigcup_{n\geq1} \Per_n(f)$ denotes the set of all periodic points of $f$. If $n=1$, then we say that $x$ is a fixed point of $f$. In the following, we always assume that $n$ is the {\em prime period} of $x$, i.e., $n$ is the smallest index so that $x\in\Per_n(f)$.
For $x\in \Per_n(f)$, we call $\tau_x=x_1\cdots x_{n}$ the {\em generating segment} of  $x$, that is $x=\Or(\tau_x)$.
Let $k\in \bN$ be fixed. We define the {\em $k$-cylinder support} of $x\in \Per_n(f)$ by 
\begin{equation}\label{defkcyl}
\cS_k(x)=\{\cC_k(f^i(x)): i\in \bN\cup\{0\}\}=\{\cC_k(f^i(x)): i= 0,\dots,n-1\}.
\end{equation}
Moreover, we say that $x\in \Per_n(f)$  is a {\em $k$-elementary periodic point}  if $\cC_k(f^i(x))\not=\cC_k(f^j(x))$ for all $i,j=0,\cdots,n-1$ with $i\not=j$. When $k=1$, we simply say that $x$ is an {\em elementary periodic point}.  We denote the set of all $k$-elementary periodic points by $\EPer^k(f)$.  We recall that $m_c(k)$ denotes the cardinality of the set of cylinders of length $k$ in $X$.  Then, it follows that the period of any $k$-elementary periodic point is at most $m_c(k)$, and, thus, $\EPer^k(f)$ is finite. 
For $x\in \Per_n(f)$, we denote the unique invariant measure supported on the orbit of $x$ by $\mu_x$, that is $\mu_x=\frac{1}{n}\sum_{i=0}^{n-1} \delta_{f^i(x)}$. For $\phi\in C(X,\bR)$, we obtain the formula
\begin{equation}\label{muxphi}
\mu_x(\phi)=\frac{1}{n}\sum_{i=0}^{n-1} \phi(f^i(x)).
\end{equation}
A particular class of shift maps are SFTs.
These shift maps can be defined as follows: Suppose $A$ is a $d\times d$ matrix with values in $\{0,1\}$, then consider the set of sequences given by  
$X=X_A=\{x\in \Sigma_d: A_{x_n,x_{n+1}}=1\}.$  The set $X_A$ is a closed (and, therefore, compact) $f$-invariant set, and we say that $f|_{X_A}$ a {\em subshift of finite type (SFT)}.

We say  $f$ is transitive if it has a dense orbit. In particular, a SFT with transition matrix $A$ is transitive if and only if $A$ is irreducible, that is, for each $i$ and $j$, there exists an $n\in \bN$ such that $A^n_{i,j} > 0.$

\subsection{Topological pressure, ground states and zero-temperature measures.}
In this section, we briefly recall the relevant facts about the topological pressure; for more details, see, e.g., \cite {Wal:81}. Let $f:X\to X$ be a transitive SFT. The topological pressure of $\phi\in C(X,\bR)$ is defined as
 \begin{equation}\label{varprinciple}
 \Ptop(\phi)=\sup_{\mu\in \cM}\left(\h_\mu(f)+\mu(\phi)\right),
 \end{equation}
 where $\h_\mu(f)$ denotes the measure-theoretic entropy of $\mu$. Moreover, $\htop(f)=\Ptop(0)$ denotes the {\em topological entropy} of $f$.
We recall that if $\nu\in\cM$ satisfies $\Ptop(\phi)=h_\nu(f)+\nu(\phi)$, then $\nu$ is an \emph{equilibrium state} of $\phi$. We denote the set of equilibrium states of $\phi$ by $\ES(\phi)$. Since the entropy map $\nu\mapsto h_\nu(f)$ is upper semi-continuous, $\ES(\phi)$ is nonempty. Furthermore, $\ES(\phi)$ is a compact and convex set whose extreme points are the ergodic equilibrium states. 

We say  $\mu\in\cM$ is a {\em ground state} of the potential $\phi$ if there exists a sequence $\beta_n\to\infty$ and equilibrium states $\mu_{n}\in \ES(\beta_n\phi)$ such that $\mu=\lim_{n\to\infty}\mu_n$.  Here, we think of $\beta$ as the inverse temperature of the system, see the discussion in Section \ref{sec:results}. Thus, ground states are accumulation points of equilibrium states as the temperature approaches zero.  We denote the set of ground states of $\phi$ by $\GS(\phi)$.  By compactness, $\GS(\phi)$ is nonempty.

Next, we consider the case where $\beta\phi$ has a unique equilibrium state $\mu_\beta=\mu_{\beta\phi}$ for all $\beta\geq 0$.
This case occurs, for example, when $\phi$ is H\"older continuous.
We say $\mu_{\infty,\phi}\in \cM$ is the 
{\em zero-temperature measure} of $\phi$ if $\mu_{\infty,\phi}=\lim_{\beta\to \infty} \mu_{\beta}$.  We note that, in general, the uniqueness of the equilibrium states of $\beta \phi$ does not guarantee the existence of the zero-temperature measure, see, e.g., \cite{BGT,CH,CR}.

\subsection{Entropy of ground states}\label{sec: ground states}
We continue to use the definitions from Section \ref{sec:results}.
Let $\phi\in C(X,\bR)$. 
For  $w\in I(\phi)$, we define
$$
\cH(w)=\cH_\phi(w)=\sup \left\{h_\mu(f): \mu(\phi)=w\right\}
$$
to be the {\em localized entropy} at $w$, see, e.g., \cite{Je,KW1}. Since $\nu\mapsto h_\nu(f)$ is affine and upper semi-continuous on $\cM$, we conclude that $ \cH$ is concave and upper semi-continuous and, therefore, continuous.  We recall that $\cH(b_\phi)$  coincides with the residual entropy $h_{\infty,\phi}$ of the potential $\phi$.
We make use of the following two lemmas to understand the behavior of the entropy as $\beta\rightarrow\infty$:
\begin{lemma}\label{lem1}
Let $(\beta_n)_n$ with $ \beta_n\in\bR^+$ be a strictly increasing sequence converging to $\infty$. Then, for any sequence of measures $(\mu_n)_n$, where $\mu_n\in \ES(\beta_n\phi)$, we have:
\begin{enumerate}[(a)]
\item $b_\phi-\mu_n(\phi)\leq \htop(f)/ \beta_n.$  Moreover, the sequence $(\mu_n(\phi))_n$ is increasing with $\lim_{n\to\infty} \mu_n(\phi)=b_\phi$;
\item $(h_{\mu_n}(f))_n$ is decreasing with $\lim_{n\to\infty} h_{\mu_n}(f)=h_{\infty,\phi}$; and
\item If $\mu\in \GS(\phi)$, then $h_\mu(f)=h_{\infty,\phi}$.
\end{enumerate}
\end{lemma}
\begin{proof}
Let $\mu\in\cM$ such that $\mu(\phi)=b_\phi$ and $h_{\mu}(f)=\cH(b_\phi)$.  Since $\mu_n$ is an equilibrium state of $\beta_n\phi$, it follows that $h_{\mu_n}(f)+\beta_n\mu_n(\phi)\geq h_{\mu}(f)+\beta_n b_\phi$.  Therefore, $\beta_n(b_\phi-\mu_n(\phi))\leq h_{\mu_n}(f)\leq\htop(f)$ and the first and last parts of Statement $(a)$ follow.

For the remaining part of Statement $(a)$, we observe that since $\mu_n$ and $\mu_{n+1}$ are equilibrium states for $\beta_n\phi$ and $\beta_{n+1}\phi$, respectively, it follows that $h_{\mu_n}(f)+\beta_n\mu_n(\phi)\geq h_{\mu_{n+1}}(f)+\beta_n\mu_{n+1}(\phi)$ and $h_{\mu_{n+1}}(f)+\beta_{n+1}\mu_{n+1}(\phi)\geq h_{\mu_{n}}(f)+\beta_{n+1}\mu_{n}(\phi)$.  Eliminating the entropies from these inequalities leads to $(\beta_{n+1}-\beta_n)(\mu_{n+1}(\phi)-\mu_n(\phi))\geq0$.  Since the $\beta_n$'s are strictly increasing, the final part of Statement $(a)$ follows.

The first part of Statement $(b)$ follows directly from Statement $(a)$ and the inequality $h_{\mu_n}(f)+\beta_n\mu_n(\phi)\geq h_{\mu_{n+1}}(f)+\beta_n\mu_{n+1}(\phi)$.  For the second part of Statement $(b)$, since $\mu_n$ is an equilibrium state of $\beta_n\phi$, it follows that $h_{\mu_n}(f)=\cH(\mu_n(\phi))$.  We recall that $\cH(b_\phi)=h_{\infty,\phi}$.  Then, by Statement (a) and the continuity of $\cH$, the second part of Statement $(b)$ follows.  

Finally, Statement $(c)$ follows from Statement $(b)$ and the upper semi-continuity of the entropy map.
\end{proof}

\begin{lemma}\label{lem2a}Let $\phi_0\in C(X,\bR)$ with $h_{\infty,\phi_0}=0$. Then, $\phi\mapsto h_{\infty,\phi}$ is continuous at $\phi_0$.
\end{lemma}
\begin{proof}
We recall that for all $\phi\in C(X,\bR)$, $\cH(b_\phi)=h_{\infty,\phi}$.  Fix $\epsilon>0$; by the continuity of $\cH$, when $\mu(\phi_0)$ is sufficiently close to $b_\phi$, then $h_{\mu}(f)\leq\cH(\mu(\phi_0))<\epsilon$.  Moreover, if $\|\phi-\phi_0\|_\infty<\delta$, then, for all $\mu\in\cM$, $|\mu(\phi)-\mu(\phi_0)|<\delta$, and, in particular, $|b_\phi-b_{\phi_0}|<\delta$.  Hence, if $\mu(\phi)=b_\phi$, then $|\mu(\phi_0)-b_{\phi_0}|<2\delta$.  Therefore, if $\delta$ is sufficiently small, then $h_{\mu}(f)<\epsilon$.  Then, by the definition of $\cH$, it follows that $h_{\infty,\phi}<\epsilon$, and the result follows.
\end{proof}

We show in Proposition \ref{prop0acomp} that the converse to Lemma \ref{lem2a} holds, i.e., that $\phi\mapsto h_{\infty,\phi}$ is continuous at $\phi_0$ if and only if $h_{\infty,\phi_0}=0$.

 \subsection{Locally constant potentials}

Let $f:X\to X$ be a transitive SFT over the alphabet $\cA$ with transition matrix $A$. For $\phi\in C(X,\bR)$ and  $k\in\bN$, we define $\var_k(\phi)=\sup\{|\phi(x)-\phi(y)|: x_1=y_1,\cdots, x_k=y_k\}.$
We say $\phi$ is constant on cylinders of length $k$ if $\var_k(\phi)=0$. We observe that $\phi$ is locally constant if and only if $\phi$ is constant on cylinders of length $k$ for some $k\in\bN$. We denote the set of all $\phi$ that are constant on cylinders of length $k$ by $LC_k(X,\bR)$.  For the remainder of this section, we let $\phi\in LC_k(X,\bR)$.
The next result shows that the case with $k\geq1$ can be reduced to the case $k=1$.

\begin{proposition}\label{propone}
 Let $k\in \bN$, $\phi\in LC_k(X,\bR)$, and  $d'=m_c(k)$. There exists a SFT $g:Y\to Y$ and homeomorphism $h:X\to Y$ with the following properties:
 \begin{enumerate}[(a)]
 \item The subshift $g$ has alphabet $\cA'=\{0,\dots,d'-1\}$ and transition matrix $A'$ such that $A'$ has at most $d$ non-zero entries in each row,
 \item The function $h$ conjugates $f$ and $g$, 
 \item The potential $\phi'=\phi\circ h^{-1}:Y\to \bR$  is constant on cylinders of length one, and
 \item $I(\phi')=I(\phi)$ and $\cH_{\phi'}=\cH_\phi$.
 \end{enumerate}
\end{proposition}

This result is  standard  and can be found in \cite{BSW}, for example.  In this paper, we only use the definitions of the subshift $g$ and the conjugating map $h$.  Let $\{\cC_k(0),\dots,\cC_k({m_c(k)-1})\}$ denote the set of cylinders of length $k$ in $X$, which we identify with $\cA'=\{0,\dots, d'-1\}$.  The transition matrix $A'$  is defined as follows: $a'_{i,j}=1$ if and only if there exists $x\in X$ with $\cC_k(x)=i$ and $\cC_k(f(x))=j$.   Finally, for $x\in X$, we define $h(x)=y=(y_n)_{n=1}^\infty$ by $y_n=\cC_k(f^{n-1}(x))$.

We observe that since the periodic point measures are dense in $\cM$, see, e.g., \cite{Par}, $I(\phi)$ can be written in terms of the periodic points, i.e., 
\begin{equation}
I(\phi)=\overline{\conv\left\{\mu_x(\phi): x\in \Per(f)\right\}}\label{perconv},
\end{equation}
where $\mu_x$ is evaluated using Formula (\ref{muxphi}) for $x\in\Per_n(f)$.
We observe that by Proposition \ref{propone}, for $x\in\Per_n(f)$, $h(x)\in \Per_n(g)$ and $\mu_x(\phi)=\mu_{h(x)}(\phi')$.  We observe that the $k$-elementary periodic points of $X$ correspond to elementary periodic points in $Y$.  Therefore, either $h(x)$ is an elementary periodic point or the generating sequence of (an iterate of) $h(x)$ can be written as a concatenation of the generating sequences of two periodic points $z_1$ and $z_2$, see \cite{Je} for details.  It follows that
$
\mu_{h(x)}(\phi')
$
is a convex combination of 
$
\mu_{z_i}(\phi')
$
for $i=1,2$ where the coefficients are the relative lengths of $z_1$ and $z_2$.  By induction, we find that $\mu_{h(x)}(\phi')$ is a convex combination of $\phi'$-integrals of $g$-elementary periodic point measures.  Since these elementary periodic points correspond to $k$-elementary periodic points of $X$, it follows from Equation \refeq{perconv}, that  
\begin{equation}\label{eperconv}
I(\phi)=\conv\left\{\mu_x(\phi): x\in \EPer^k(f)\right\}.
\end{equation}
We note that the closure is not needed in Equation \eqref{eperconv} since $\EPer^k(f)$ is finite; therefore, the closure can also be omitted in Equation \eqref{perconv}.

Next, we characterize the decomposition of $LC_k(X,\bR)=\cO_k\,\dot\cup\,\cU_k\,\dot\cup\,\cV_k$ from Equation (\ref{eq:decompositionk}) in terms of the number and behavior of the elementary periodic points which achieve the maximum value in $I(\phi)$.  We define
$$
\EPer_{\rm max}^k(\phi)=\left\{x\in \EPer^k(f): \mu_x(\phi)=b_\phi\right\}.
$$
\begin{definition}
Let $f:X\to X$ be a transitive SFT.
\begin{enumerate}[(a)]
\item Then $\phi\in\cO_k$ if $\EPer_{\rm max}^k(\phi)$ contains a single $k$-elementary periodic orbit. 
\item Furthermore, $\phi\in\cU_k$ if $\EPer_{\rm max}^k(\phi)=\{x^1,\dots,x^\ell\}$ contains more than one $k$-elementary periodic orbit and the $k$-cylinder support of different $k$-elementary periodic orbits are distinct.  In other words, $\cC_k(x^i)\not=\cC_k(x^j)$ for all $x^i,x^j\in \EPer_{\rm max}^k(\phi)$ with $i\not=j$.
\item Finally,  $\phi\in\cV_k$ if $\EPer_{\rm max}^k(\phi)=\{x^1,\dots,x^\ell\}$ contains more than one $k$-elementary periodic orbit and the $k$-cylinder support of different  $k$-elementary  periodic orbits are not distinct.  More precisely, $\cC_k(x^i)=\cC_k(x^j)$ for some $i,j\in\{1,\dots,\ell\}$ with $i\not=j$.
\end{enumerate}
\end{definition}
As mentioned in Section 1.2, the properties of the partition  $LC_k(X,\bR)=\cO_k\,\dot\cup\,\cU_k\,\dot\cup\,\cV_k$
follow from \cite{WY}. Furthermore, \cite{WY} implies  the following:
\begin{proposition}Let $f:X\to X$ be a transitive SFT.  Then
\begin{enumerate}[(a)]
\item If $\phi\in \cO_k$ and $x\in\EPer_{\rm max}^k(\phi)$, then $\mu_{\infty,\phi}=\mu_x$ and $h_{\infty,\phi}=0$;
\item If $\phi\in \cU_k$, then $h_{\infty,\phi}=0$ and $\mu_{\infty,\phi}$ is a convex combination of the periodic point measures corresponding to periodic orbits in $\EPer_{\rm max}^k(\phi)$; and
\item If $\phi\in \cV_k$, then $\mu_{\infty,\phi}$ is a measure of maximal entropy of a non-discrete (and not necessarily transitive) SFT $X_{\rm max}\subset X$ and $h_{\infty,\phi}>0$.  
\end{enumerate}
\end{proposition}

\subsection{Basics from computability theory}\label{sec:compute:basic}

Computability theory provides information about the feasibility and accuracy of computational experiments when using approximate data.  For instance, a computable function is one in which the results include explicit error bounds on the accuracy of the value of the function.  Without such an accuracy guarantee, a computer experiment might miss or misinterpret interesting behaviors.  In this section, we discuss the basic ideas from computability theory which are needed in this paper.

We are interested in the feasibility of computational experiments on $I(\phi)$, $\mu_{\infty,\phi}$, and $h_{\infty,\phi}$.  To this end, we focus only on the definitions that are needed for these particular objects.  For a more thorough discussion of topics from computability theory, see, e.g., \cite{BBRY,BHW, BY,BSW,GHR,RW,RYSurvey,K} and the references therein.  We use different, but closely related, definitions to those in \cite{BY} and \cite{GHR}, see also \cite{BSW}.  Throughout this discussion, we use a bit-based computation model, such as a Turing machine, as opposed to a real RAM (random access machine) model \cite{S} (where these questions are trivial).  One can think of the set of Turing machines as a particular countable set of functions.  We denote the output of a Turing machine $\psi$ on input $x$ by $\psi(x)$.
\begin{definition}[cf {\cite[Definition 1.2.1]{BY}}]
Let $x\in\bR^m$.  An {\em oracle} approximating $x$ is a function $\psi$ such that on input $n\in \bN$, $\psi(n)\in\bQ^m$ with $\|\psi(n)-x\|<2^{-n}$.  Moreover, $x$ is {\em computable} if there is a Turing machine $\psi$ which is an oracle for $x$.
\end{definition}

Since there are only countably many Turing machines, there are only countably many computable points in $\bR^m$.  The computable points in $\bR$ include the rational and algebraic numbers as well as some transcendental numbers, such as $e$ and $\pi$.  Since points in $\bR^{m_c(k)}$ are in bijective correspondence with locally constant potentials $LC_k(X,\bR)$, the definition of computability also carries over to these potentials.  

The main results in this paper are about the computability of certain functions.  We now provide the definition of a computable function.
\begin{definition}
Let $S\subset\bR^m$.  A function $g:S\rightarrow\bR$ is {\em computable} if there is a Turing machine $\chi$ so that for any $x\in S$ and any oracle $\psi$ for $x$, $\chi(\psi,n)$ is a rational number so that $|\chi(\psi,n)-g(x)|<2^{-n}$.
\end{definition}

We observe that, in this definition, $x$ does not need to be computable, i.e., the oracle $\psi$ does not need to be a Turing machine.  When $x$ is computable, then $g(x)$ is also computable since $\chi(\psi,\cdot)$ is an oracle Turing machine for $g(x)$.  Additionally, we observe that computability of a function is defined in terms of the supremum norm.  Since the supremum norm does not generate the same topology as the H\"older or Lipschitz norms, previous results on the H\"older and Lipschitz norms cannot be applied to this paper, see, e.g., \cite{C,CLT,M,QS}

The composition of computable functions is computable because the output of one Turing machine can be used as the input approximation for subsequent machines.  In addition, basic operations, such as the arithmetic operations and the minimum and maximum functions are computable, see \cite{BHW} for more details on these, and related properties.

We also note that the definition of a computable function uses any oracle for $x$ and applies even when $x$ is not computable.  Therefore, we can conclude that for any sufficiently close approximation $y$ to $x$, $g(y)$ approximates the value of $g(x)$, i.e., $g$ is continuous.  We make this property explicit in the following lemma:

\begin{lemma}[cf {\cite[Theorem 1.5]{BY}}]
Let $S\subset \bR^m$ and $g:S\rightarrow\bR$.  If $g$ is computable, then $g$ is continuous.
\end{lemma}

In this paper, we include functions which have a weaker notion of computability called upper semi-computability.  In this case, the convergence and accuracy of the approximations are weaker.

\begin{definition}[cf {\cite[Definition 2.7]{GHR}}]
Let $S\subset\bR^m$.  A function $g:S\rightarrow\bR$ is {\em upper semi-computable} (also called right recursively enumerable or right computable) if there is a Turing machine $\chi$ so that for any $n\in \bN$ and $x\in S$, and any oracle $\psi$ for $x$, $\chi(\psi,n)$ is a rational number with the following properties: The sequence $(\chi(\psi,n))_n$ is nonincreasing and $\lim_{n\rightarrow\infty}\chi(\psi,n)=g(x)$.
\end{definition}

Computable functions are also semi-computable, but there are functions which are semi-computable, but not computable, see, e.g., \cite{BY}.  The main distinction between computable and semi-computable functions is an error estimate.  In fact, a semi-computable function with an error estimate is computable.  By mirroring the definition above, we may define lower semi-computable functions.  We observe that a function is computable if and only if it is both upper- and lower-semi computable.  

In this paper, we frequently consider functions $g:S\rightarrow\bR$ which are not computable.  The non-computability of a function is based on its entire domain.  In other words, $g$ may have some computability properties at some points of its domain (but not all).  For the purposes of this paper, we introduce the notion of a function which computable at a point.  This notion of computability is a local property and is a computable version of continuity.

\begin{definition}
Let $S\subset\bR^m$ be an open set and let $x\in S$.  A function $g:S\rightarrow\bR$ is {\em computable at $x$} if there exists a Turing machine $\chi$ so that for any oracle $\psi$ for $x$, $\chi(\psi,n)$ is a rational number with the following property: Let $\ell_n$ be the highest precision to which the oracle $\psi$ is queried within $\chi$, then, for all $y\in S$ such that there exists an oracle $\psi'$ for $y$ that agrees with $\psi$ up to precision $\ell_n$, i.e., $\chi(\psi,n)=\chi(\psi',n)$, we have $|\chi(\psi',n)-g(y)|<2^{-n}$.
\end{definition}

Roughly speaking, this definition states that if $y$ is close enough to $x$ so that, up to precision $\ell_n$, $\psi$ could be an oracle for $y$, then $\chi(\psi,n)$ is a good approximation for $g(y)$.  This definition is a computable version of continuity as the oracle condition is similar to a $\delta$-neighborhood of $x$.  We observe that this definition does not include a decidability statement, e.g., this definition is existential.  In other words, we do not assume that there exists a Turing machine decides whether $x\in S$ is a computable point.  Additionally, we note that if $T\subset S$ and $g|_T$ is computable, then for every $x$ in the interior of $T$, $g$ is computable at $x$; thus, this condition is necessary for all points in the domain of a computable function on an open set.  We note that there are other potential notions of computability at a point that capture other computability properties of $g$, but these other notions may not be as closely related to the results in this paper.  We leave the details to the interested reader.

Finally, since one of our main theorems involves recursively open sets, we include the definition of a recursively open set.
\begin{definition}[{cf \cite[Definition 2.4]{GHR}}]
Let $S\subset\bR^m$ be an open set.  $S$ is a {\em recursively open set} (also called a semi-decidable set or a lower-computable set) if there exists a Turing machine $\psi$ such that $\psi$ produces a (possibly infinite) sequence of pairs $(z_i,n_i)$ so that $z_i\in\bQ^m$ is a rational vector and $n_i\in\bZ$ so that
$$
S=\bigcup_i B\left(z_i,2^{-n_i}\right).
$$
\end{definition}
A recursively open set is one for which there exists a Turing machine that terminates on input $s$ if $s\in S$ and does not terminate if $s\not\in S$.  Therefore, we observe that we cannot decide, using such a Turing machine, if $s\not\in S$ as it is impossible to decide if the Turing machine will run forever or has not run long enough.

\subsection{Computability theory for SFTs}

Since our main results pertain to SFTs, in this section, we specialize computability theory to this case.  For further details on computability for SFTs, see, e.g., \cite{BSW}.  Throughout this section, we assume that $X$ is a SFT.  We begin by adapting the definition of a computable point to SFTs.

\begin{definition}
Let $x\in X$.  An {\em oracle} for $x$ is a function $\psi$ such that for any natural number $n$, $\psi(n)=x_n$.  Moreover, $x$ is {\em computable} if there is a Turing machine $\psi$ which is an oracle for $x$.
\end{definition}

We note that all periodic and preperiodic points of $X$ are computable.  In fact, there is a Turing machine that produces (over an infinite amount of time) a list of all preperiodic points of $X$ since a preperiodic point corresponds to a pair of finite sequences: the prefix and the periodic part.  Thus, there exists a Turing machine that lists all pairs of finite sequences in the alphabet of $X$ and checks each sequence against the allowable transitions for $X$.

We also extend the notion of computable functions to the case of SFTs in the following definition:

\begin{definition}
Let $S\subset X$.  A function $g:S\rightarrow\bR$ is {\em computable} if there exists a Turing machine $\chi$ such that for any $x\in S$ and oracle $\psi$ for $x$, $\chi(\psi,n)$ is a rational number so that $|\chi(\psi,n)-g(x)|<2^{-n}$.  In this paper, we also consider functions whose domains are subsets  of $C(X,\bR)$.  In particular, for $S\subset C(X,\bR)$, we say $h:S\rightarrow\bR$ is computable if there is a Turing machine $\eta$ so that for any function $\phi\in S$ and oracle $\chi$ for $\phi$, $\eta(\chi,n)$ is a rational number with $|\eta(\chi,n)-h(\phi)|<2^{-n}$.
\end{definition}

The notion of an upper semi-computable function carries over similarly and we leave the details of the formulation to the interested reader.  We note that if $\theta$ from the Tychnov product topology is a computable real number, then the function for the distance between two points of $X$ is a computable function and $X$ is a computable metric space, cf {\cite[Definition 2.2]{GHR}}.

We recall that we may identify $LC_k(X,\bR)$ with $\bR^{m_c(k)}$.  The following result is a standard tool when dealing with the computability of potentials for SFTs, see, e.g., \cite{BSW}:
\begin{lemma}\label{lemabLC}
There exists a Turing machine, which, given input $n\in\bN$ and an oracle $\chi$ of $\phi\in C(X,\bR)$, produces $m_n\in\bN$ and $\phi_n\in LC_{m_n}(X,\bQ)$ such that $\|\phi-\phi_n\|_\infty<2^{-n}$.
\end{lemma}

We observe that when $\phi\in LC_k(X,\bR)\setminus LC_{k-1}(X,\bR)$, the construction of Lemma \ref{lemabLC} could produce a $\phi_n$ with $m_n\gg k$.  For example, fix $\ell>k$ and consider a nonempty cylinder $\cC_k(x)$ of $X$.  Suppose that the partition of $\cC_k(x)$ into cylinders of length $\ell$ has at least two nonempty cylinders of length $\ell$ (this will be the case when $\ell$ is sufficiently large and $X$ is transitive and has positive topological entropy).  Perturbing the value of $\phi$ on one of the cylinders of length $\ell$ results in a potential $\phi_n$ where $\|\phi-\phi_n\|_\infty$ is arbitrarily small, while $m_n\geq\ell$ is arbitrarily large.  There is, however, a procedure to prevent $m_n$ from growing arbitrarily large.  We describe this procedure in the following result:

\begin{lemma}
There exists a Turing machine, which, given input $n\in\bN$ and an oracle $\chi$ of $\phi\in LC(X,\bR)$ produces $\ell_n\in\bN$ and $\widetilde{\phi}_n\in LC_{\ell_n}(X,\bR)$ with the following properties:
\begin{enumerate}
\item $\|\phi-\widetilde{\phi}_n\|_{\infty}<2^{-n}$ and
\item If $\var_j(\phi)=0$, then $\ell_n\leq j$.
\end{enumerate}
\end{lemma}
\begin{proof}
Let $\phi_n$ and $m_n$ be produced as in Lemma \ref{lemabLC}.  We begin by considering the cylinders of length $m_n$.  In \cite[Section 5.2]{BSW}, it is shown that there is a Turing machine which considers each cylinder $\cC_{m_n}(\tau)$ and either finds a point $x\in\cC_{m_n}(\tau)\cap X$ or reports that the cylinder is empty.  Therefore, we may approximate the values that $\phi_n$ attains to any precision.

For each $1\leq i\leq m_n$, we can also consider all cylinders of length $i$.  We observe that each nonempty cylinder of length $i$ can be partitioned into a finite collection of cylinders of length $m_n$.  Since the maximum and minimum over a finite set is computable, we can find an upper bound on the variation $\var_i(\phi_n)$ with error at most $2^{-n+1}$.  Let $\ell_n$ be the smallest value of $i$ where the variation is bounded above by $2^{-n+2}$.  We construct $\widetilde{\phi}_n\in LC_{\ell_n}(X,\bR)$ by combining all cylinders of length $\ell_n$, and, for each nonempty cylinder $\cC_{\ell_n}(\widetilde{\tau})$ of length $\ell_n$, assigning the value of $\phi_n(\cC_{m_n}(\tau))$ to $\widetilde{\phi}_n(\cC_{\ell_n}(\widetilde{\tau}))$ for some arbitrary nonempty cylinder of length $m_n$ contained in $\cC_j(\widetilde{\tau})$.

We observe that if $\var_j(\phi)=0$, then $\ell_n\leq j$ as follows:  By the construction of $\phi_n$, the variation in a $j$-cylinder is at most $2^{-n+1}$, so the upper bound on the variation is at most $2^{-n+2}$.  Additionally, if $\phi\in LC_k(X,\bR)\setminus LC_{k-1}(X,\bR)$, we can show that for $n$ sufficiently large, $\ell_n=k$ as follows: The variation of $\phi$ on $(k-1)$-cylinders is bounded away from zero since $\phi\not\in LC_{k-1}(X,\bR)$, so, when $n$ is sufficiently large, the variation bound of $2^{-n+2}$ is small enough to distinguish $(k-1)$-cylinders.
\end{proof}

Computability can also be extended to the space $\cM$ of invariant measures on $X$.  In this case, approximations to measures are given by convex combinations of Dirac measures.  Moreover, we use the Wasserstein-Kantorovich distance, which generates the weak$^\ast$ topology and is defined by
$$
W_1(\mu_1,\mu_2)=\sup_{\phi\in1\text{-Lip(X)}}\left|\mu_1(\phi)-\mu_2(\phi)\right|
$$
for all $\mu_1,\mu_2\in \cM$ (where $\text{1-Lip(X)}$ denotes the space of Lipschitz continuous functions on $X$ with Lipschitz constant $1$).  It is shown in \cite{GHR} that this distance is computable.

\section{Upper semi-computability of the residual entropy}\label{sec:3}
The goal of this section is to prove Theorem A. We start with a discussion of the pressure function for continuous potentials. These results are fairly standard in the H\"older continuous case, but they are more challenging for potentials which are only continuous.  The difficulties arise from the lack of uniqueness results for equilibrium states and, in particular, the possibility of phase transitions.  To overcome these challenges, we make use of several tools, including methods from convex analysis, see, e.g., \cite{R}. 

Let $\phi:X\to\bR$ be a fixed continuous potential. We note that we do not  assume the uniqueness of the equilibrium states. We call $\beta\mapsto P(\beta)\eqdef \Ptop(\beta\phi)$ the pressure function of $\phi$. The pressure function is convex, see, e.g., \cite{Wal:81}, and, thus, it has left and right derivatives 
$$
\partial_\pm P(\beta)=\lim_{h\to 0^\pm}\frac{ P(\beta+ h)-P(\beta)}{h}.
$$
 Moreover, since $\mu\mapsto h_\mu(f)$ is upper semi-continuous, it follows from \cite[Proposition 1]{Je} and \cite[Lemma 1]{Wal92} that
\begin{equation}\label{esal1}
\partial_-P(\beta)=\min_{\mu\in \ES(\beta\phi)} \mu(\phi)\quad\text{and}\quad\partial_+P(\beta)=\max_{\mu\in \ES(\beta\phi)} \mu(\phi).
\end{equation}
Furthermore, since $\ES(\beta\phi)$ is a compact and convex subset of $\cM$, for all $\partial_-P(\beta)\leq \alpha\leq \partial_+P(\beta)$, there exists $\mu_\alpha\in \ES(\beta\phi)$ with $\mu_\alpha(\phi)=\alpha$. In particular, the minimum and maximum in Equation \eqref{esal1} is well-defined. 
We observe that $\beta\mapsto P(\beta)$ is differentiable at $\beta$ if and only if $I_{\beta}\eqdef\{\mu(\phi):\mu\in \ES(\beta\phi)\}$ is a singleton\footnote{We note that the nondifferentiability points of the pressure function are phase transitions, i.e., points of coexistence of multiple equilibrium states where each ergodic equilibrium state represents a phase.}.     
 Moreover,
\begin{equation}\label{eqJencor2}
\inn I(\phi)=(a_\phi,b_\phi)\subset \bigcup_{\beta\in \bR} I_\beta,
\end{equation}
see \cite[Corollary 2]{Je}.
Since   $\beta\mapsto P(\beta)$ is convex, it is differentiable on $\bR$ with the exception of at most countably many points $\beta\in \bR$. We define 
$$
h_{\rm max}(\beta)=\max_{\mu\in \ES(\beta\phi)} h_\mu(f) \quad\text{and}\quad h_{\rm min}(\beta)=\min_{\mu\in \ES(\beta\phi)} h_\mu(f).$$
Thus, Equation \eqref{varprinciple} yields
\begin{equation}\label{eqnet}
P(\beta)=h_{\max}(\beta)+\beta \partial_-P(\beta)=h_{\min}(\beta)+\beta \partial_+P(\beta).
\end{equation}
Moreover, the  convexity of the pressure function implies 
\begin{equation}\label{eqmon}
\partial_+P(\beta_1)\leq \partial_-P(\beta_2)\quad \text{and}\quad h_{\rm min}(\beta_1)\geq h_{\max}(\beta_2)
\end{equation}
whenever $\beta_1<\beta_2$.  First, we consider the case when $P$ is differentiable at $\beta$. In this case, Equation \eqref{eqnet} becomes
\begin{equation}\label{eqpresdiff}
P(\beta)=h(\beta)+\beta \partial P(\beta),
\end{equation}
where $h(\beta)\eqdef h_{\max}(\beta)=h_{\min}(\beta)$ and $\partial P(\beta)\eqdef \partial_-P(\beta_1)=\partial_+P(\beta_1)$.  Using these functions, we develop a series of results to study the computability of the pressure and entropy functions.

\begin{proposition} \label{proholder}
Suppose that $\phi\in C(X,\bR)$ is given by an oracle.  Let $B$ be the set of points where the function $\beta\mapsto P(\beta)$ is differentiable.  Then, the functions $\beta\mapsto\partial P(\beta)$ and $\beta\mapsto h(\beta)$ are computable on $B$.
\end{proposition}
\begin{proof}
For $n\in\bN$, we consider the left and right difference quotients of $P$ at $\beta$,
$$
p_\pm(n)=\frac{P(\beta\pm 1/n)-P(\beta)}{\pm 1/n}.
$$
Since the topological pressure is computable, see \cite{BSW, Sp}, we can compute  $p_\pm(n)$  to any given accuracy.
On the other hand, since $P$ is convex, when $P$ is differentiable at $\beta$, $p_\pm(n)$ converges from above and below to $\partial P(\beta)$ as $n\to\infty$, respectively.  When the upper and lower bounds are close enough, any point between them can be used to approximate $\partial P(\beta)$ to arbitrary precision.  Finally, the computability of $h$ follows from Equation \eqref{eqpresdiff}.
\end{proof}
Proposition \ref{proholder} shows that if $\phi$ is a H\"older continuous potential given by an oracle, then the functions $\beta\mapsto h(\beta)$ and $\beta\mapsto \partial P(\beta)$ are computable. Combining this observation with Lemma \ref{lem1}, we conclude that Theorem A holds for H\"older continuous potentials. To prove the general case, 
we make use of the following result to include the possibility of phrase transitions:

\begin{proposition}\label{lem2}
Suppose $\phi\geq 0$ and let $0\leq\beta_1<\beta_2$. We define $\alpha=\alpha(\beta_1,\beta_2)=(P(\beta_2)-P(\beta_1))/(\beta_2-\beta_1)$. Then, there exist $\beta_1< \beta< \beta_2$ 
and $\mu\in\ES(\beta\phi)$ such that
\begin{equation}\label{eqentab}
P(\beta_1)-\beta_2 \alpha\leq h_\mu(f) \leq P(\beta_2)-\beta_1 \alpha.
\end{equation}
\end{proposition}
\begin{proof}
First, we observe that since $\phi\geq 0$, the map $\beta\mapsto P(\beta)$ is increasing. If
$\partial_+P(\beta_1)=\partial_-P(\beta_2)$, then $h(\beta)$ and $\partial P(\beta)$ are constant for $\beta_1< \beta<\beta_2$, so $P\vert_{(\beta_1,\beta_2)}$ is an affine function of $\beta$.  Moreover, for all $\beta_1< \beta< \beta_2$, $\partial P(\beta)=\alpha$.  Finally, combining this with Equation (\ref{eqnet}), it follows that $P(\beta_2)-\beta_1\alpha=h(\beta)+(\beta_2-\beta_1)\partial P(\beta)\geq h(\beta)$ and $P(\beta_1)-\beta_2\alpha=h(\beta)-(\beta_2-\beta_1)\partial P(\beta)\leq h(\beta)$.  Therefore, Inequality \eqref{eqentab} holds for all $\beta_1< \beta< \beta_2$ and all $\mu\in \ES(\beta\phi)$.

It remains to consider the case where $\partial_+P(\beta_1)<\partial_-P(\beta_2)$. Since $\alpha$ is the slope of the line segment joining $(\beta_1,P(\beta_1\phi))$ and $(\beta_2,P(\beta_2\phi))$, the convexity  of the pressure function  implies that  $\partial_+P(\beta_1)<\alpha<\partial_-P(\beta_2)$.  Thus, by Equation \eqref{esal1}, $\alpha\in \inn I(\phi)$. It now follows from Equation \eqref{eqJencor2} that there exists  $\beta\in\bR$ and $\mu\in \ES(\beta\phi)$ such that $\mu(\phi)=\alpha$.  Moreover, by Equation \eqref{eqmon}, we may restrict $\beta$ to $\beta_1<\beta<\beta_2$.  Applying Equation \eqref{varprinciple} yields
$$h_\mu(f)=P(\beta)-\beta\mu(\phi)=P(\beta)-\beta\alpha.$$ Finally, Equation \eqref{eqentab} follows  since the pressure function is increasing.
\end{proof}

The following auxiliary lemma is used in the proofs of both Theorems A and B.  In the lemma, we show that the endpoints of $I(\phi)=[a_\phi,b_\phi]$ are computable points.

\begin{lemma}\label{lem:computableI}
The functions $\phi\mapsto a_\phi$ and $\phi\mapsto b_\phi$ are  computable on 
$C(X,\bR)$. 
\end{lemma}

\begin{proof}
We first note that the functions $\phi\mapsto a_\phi$ and $\phi\mapsto b_\phi$ are Lipschitz continuous with Lipschitz constant $1$ on $C(X,\bR)$.  Thus, by applying Lemma \ref{lemabLC} to generate $\phi_n$, we may conclude that $|a_\phi-a_{\phi_n}|<2^{-n}$ and $|b_\phi-b_{\phi_n}|<2^{-n}$.  Therefore, it is enough to prove the statement for locally constant potentials.  Let $m_n$ be the integer constructed in Lemma \ref{lemabLC}, i.e., $\phi_n\in LC_{m_n}(X,\bR)$.  Then, by Equation \refeq{eperconv}, it is enough to approximate $\mu_x(\phi)$ for all $m_n$-elementary periodic points of $X$.  We use Formula \refeq{muxphi} to approximate $\mu_x(\phi)$ to any desired precision.  Since there are only finitely many $m_n$-elementary periodic points and the maximum and minimum of a finite set are computable, we can approximate $a_\phi$ and $b_\phi$ to any desired precision.
\end{proof}

We are now ready to present the proof of Theorem A which uses the computability of the topological pressure, Lemma \ref{lem1}, and Proposition \ref{lem2}.  We begin with a technical lemma that forms the central argument of the main theorem.

\begin{lemma}\label{lem:approxentropy}
Let $\phi\in C(X,\bR)$ with $\phi\geq 0$ be given by an oracle $\psi$.  Suppose that rational numbers $\beta_1<\beta_2$ are given.  There exists a Turing machine $\chi$ so that $\chi(n,\phi)$ is a rational number such that there exists\footnote{We note that the lemma does not require the computability of $\beta$, only its existence.} a $\beta$ with $\beta_1<\beta<\beta_2$ and $\mu\in\ES(\beta\phi)$ such that $|\chi(n,\psi)-h_\mu(f)|<2^{-n}$.
\end{lemma}

\begin{proof}
We observe that since the pressure function $\beta\mapsto\Ptop(\beta\phi)$ is continuous, as $\beta_2\rightarrow\beta_1$, the upper and lower bounds of Inequality (\ref{eqentab}) approach each other.  Therefore, if we can find $\beta_1'$ and  $\beta_2'$ so that $\beta_1\leq \beta_1'<\beta_2'\leq\beta_2$ and the upper and lower bounds of Inequality (\ref{eqentab}) are within $2^{-n}$, any rational number satisfying the inequalities of Inequality (\ref{eqentab}) can be used to approximate $h_\mu(f)$.

We recall that the pressure function $\beta\mapsto\Ptop(\beta\phi)$ is computable, see \cite{BSW,Sp}.  Therefore, the upper and lower bounds in Inequality (\ref{eqentab}) are also computable.  We consider a sequence $(\beta_{1,m}',\beta_{2,m}')$ of pairs of rational numbers so that $\beta_1\leq\beta_{1,m}'<\beta_{2,m}'\leq \beta_2$ and $\beta_{2,m}'-\beta_{1,m}'$ decreases to zero as $m\to\infty$.  By approximating the upper and lower bounds of Inequality (\ref{eqentab}) sufficiently well for each $m$, we may compute an $m$ so that the upper and lower bounds of Inequality (\ref{eqentab}), when applied to $\beta_{1,m}'$ and $\beta_{2,m}'$, are within $2^{-n}$. 
\end{proof}

Next, we present the proof of Theorem A, which is broken into the following two statements:

\begin{theorem}\label{halfThmA}
The function $\phi\mapsto h_{\infty,\phi}$ is upper semi-computable on $C(X,\bR)$.
\end{theorem}
\begin{proof}
Suppose that $\phi\in C(X,\bR)$ is given by an oracle.  We can compute a lower bound $q$ on $\phi$ by using Lemma \ref{lemabLC} to approximate $\phi$ by a locally constant potential $\phi_n$ and approximating a lower bound on $\phi_n$.  We observe that $\ES(\beta\phi)=\ES(\beta(\phi+q))$ and $\phi+q\geq 0$.  By applying Lemma \ref{lem:approxentropy} to a strictly increasing sequence $(\beta_n)_n$ converging to $\infty$, we compute a sequence of entropies $h_{\mu_n}(f)$ for $\mu_n\in\ES(\beta(\phi+q))$ with $\beta_n<\beta<\beta_{n+1}$.  Then, applying Lemma \ref{lem1}, we conclude that these entropies approach the residual entropy from above.
\end{proof}
Next, we characterize the continuity of the residual entropy. 
\begin{proposition}\label{prop0acomp}
The function  $\phi\mapsto h_{\infty,\phi}$ is continuous at $\phi_0\in C(X,\bR)$ if and only if $h_{\infty,\phi_0}=0$.
\end{proposition}
\begin{proof}
Let $\phi_0\in C(X,\bR)$. If $h_{\infty,\phi_0}=0$ then $\phi\mapsto h_{\infty,\phi}$ is continuous at $\phi_0$ by Lemma \ref{lem2a}.
On the other hand, suppose $h_{\infty,\phi_0}>0$.  We recall the definition of the set of uniquely maximizing potentials $\cO=\bigcup_k \cO_k$ from Section \ref{sec:results}.  We observe that since $\cO_k$ is dense in $LC_k(X,\bR)$, it follows that $\cO$ is dense in $LC(X,\bR)$.  Since $\cO$ consists of the uniquely maximizing locally constant potentials, for all $\phi\in\cO$, $h_{\infty,\phi}=0$.  Finally, since $LC(X,\bR)$ is dense in $C(X,\bR)$, we conclude that the map $\phi\mapsto h_{\infty,\phi}$ is not continuous at $\phi_0$.
\end{proof}

We end this section by using the previous two results to prove Corollary \ref{cor:computabilityentropy}
\begin{corollary}
The function $\phi\mapsto h_{\infty,\phi}$ is computable at $\phi_0$ if and only if $h_{\infty,\phi_0}=0$.
\end{corollary}
\begin{proof}
If $h_{\infty,\phi_0}>0$, then by Proposition \ref{prop0acomp}, we know that the map $\phi\mapsto h_{\infty,\phi}$ is not continuous at $\phi_0$, so the function cannot be computable at $\phi_0$.  On the other hand, suppose that $h_{\infty,\phi_0}=0$.  Then, for any oracle $\psi$ for $\phi_0$, using Theorem \ref{halfThmA}, there is a Turing machine $\chi$ so that $(\chi(\psi,m))_m$ is a sequence of rational numbers decreasing to zero.  By taking $m_n$ sufficiently large, $\chi(\psi,m_n)<2^{-n}$.  Let $\ell_n$ be the largest precision to which the oracle $\psi$ is queried within $\chi$ and let $\phi'\in C(X,\bR)$ be a function such that there exists an oracle $\psi'$ for $\phi'$ that agrees with $\psi$ up to precision $\ell_n$.  Then $\chi(\psi,m_n)=\chi(\psi',m_n)$ computes an upper bound on $h_{\phi',\infty}$.  Since the entropy is nonnegative, $|\chi(\psi,m_n)-h_{\phi',\infty}|<2^{-n}$ and the function $\phi\mapsto h_{\infty,\phi}$ is computable at $\phi_0$.
\end{proof}

\section{Computability of Zero Temperature Measures for locally constant potentials: the case of bounded cylinder length}\label{sec:4}

In this section, we prove Theorem B by breaking the statement into a series of propositions.  Throughout this section, we assume that $\theta$ from the Tychonov product topology, see Equation \eqref{defmet}, is a computable real number.  Moreover, we assume, whenever necessary, that $\phi\in LC_k(X,\bR)$ is a potential given by an oracle.

We observe that by using Lemma \ref{lem:computableI}, we can compute a superset of $\EPer_{\max}^k(\phi)$.  In particular, for $x\in\EPer^k(f)$, we can approximate $\mu_x(\phi)$ using  Formula \eqref{muxphi}.  Then, $\EPer_{\max}^k(\phi)$ is a subset of those $k$-elementary points for which the approximations of $\mu_x(\phi)$ and $b_\phi$ permit the possibility of equality.  By increasing the accuracy of these approximations, the computed superset of $\EPer_{\max}^k(\phi)$ shrinks.  We can conclude that there are Turing machines $\psi_{\cO_k}$ and $\psi_{\cO_k\,\dot\cup\,\cU_k}$ which take a potential as input and terminate if and only if $\phi\in\cO_k$ or $\phi\in\cO_k\,\dot\cup\,\cU_k$, respectively.  More precisely, if $\phi\in\cO_k$, then, when computing with high enough precision, the computation of a superset of $\EPer_{\max}^k(\phi)$ results in a single $k$-elementary orbit.  In this case, $\EPer_{\max}^k(\phi)$ equals this unique $k$-elementary orbit.  On the other hand, if $\phi\in\cO_k\,\dot\cup\,\cU_k$, then, with high enough precision, the computation of a superset of $\EPer_{\max}^k(\phi)$ results in a collection $k$-elementary orbits which have disjoint $k$-cylinder support.  As this discussion already hints at, we now prove that $\cO_k$ and $\cO_k\,\dot\cup\,\cU_k$ are both recursively open sets.
\begin{proposition}\label{prop:recursivelyopen}
The sets $\cO_k$ and $\cO_k\,\dot\cup\,\cU_k$ are recursively open sets.
\end{proposition}
\begin{proof}
Fix a countable dense subset of $LC_k(X,\bR)\cap\overline{B}(0,1)$.  For instance, we may choose $\bQ^{m_c(k)}\cap\overline{B}(0,1)$.  We also use the Turing machines $\psi_{\cO_k}$ and $\psi_{\cO_k\,\dot\cup\,\cU_k}$ constructed above\footnote{For potentials $\phi\in\bQ^{m_c(k)}\cap\overline{B}(0,1)$, the values of $\phi$ are exact and $\EPer_{\max}^k(\phi)$ can be calculated explicitly.  We, however, do not use this fact here.}.  If $\phi\in\cO_k$, then there is a positive gap between the approximation to $b_\phi$ and the second-largest value of $\mu_y(\phi)$ for a $k$-elementary periodic point $y$.  Perturbations of $\phi$ by no more than half this gap remain within $\cO_k$.  Similarly, if $\phi\in\cO_k\,\dot\cup\,\cU_k$, then there is a gap between the approximation to $b_\phi$ and the largest $\mu_x(\phi)$ of a $k$-elementary periodic point $x$ which is not included in the superset of $\EPer_{\max}^k(\phi)$ constructed above.  Perturbations of $\phi$ by no more than half this gap remain within $\cO_k\,\dot\cup\,\cU_k$.  By using more accurate approximations, we can discover more potentials and refine the radii of the constructed balls, so that, in the limit, the constructed open sets  cover $\cO_k$ or $\cO_k\,\dot\cup\,\cU_k$.  
\end{proof}

We now discuss the computability of the entropy and the zero-temperature measure.  These propositions are the main computability statements of Theorem B.

\begin{proposition}\label{prop:computability}
The map $\phi\mapsto h_{\infty,\phi}$ is computable on $\cO_k\,\dot\cup\,\cU_k$.  Moreover, the map $\phi\mapsto \mu_{\infty,\phi}$ is computable on $\cO_k$.
\end{proposition}
\begin{proof}
Suppose that $\phi\in\cO_k\,\dot\cup\,\cU_k$.  Since the entropy of all zero-temperature measures of potentials in $\cO_k\,\dot\cup\,\cU_k$ is zero, which is computable, the entropy function is computable.  Suppose now that we know that $\phi\in\cO_k$.  By inspecting the proof of Proposition \ref{prop:recursivelyopen}, we find that for all $\phi'$ in the ball produced in the proof, the same $k$-elementary orbit $x$ maximizes $\mu(\phi')$.  Therefore, for every $\phi'$ in the ball, the zero-temperature measure is $\mu_x=\mu_{\infty,\phi}$.  This measure is computable since the supporting $k$-elementary periodic point is computable.
\end{proof}

We now complete the proof of Theorem B by showing that the functions $\phi\mapsto \mu_{\infty,\phi}$ and $\phi\mapsto h_{\infty,\phi}$ are not continuous, and, hence, not computable on the complement of the points in Proposition \ref{prop:computability}.

\begin{proposition}
The map $\phi\mapsto \mu_{\infty,\phi}$ is not continuous at any $\phi_0\in\cU_k\,\dot\cup\,\cV_k$.  Moreover, the map $\phi\mapsto h_{\infty,\phi}$ is not continuous at any $\phi_0\in\cV_k$.  In particular, the corresponding maps are not computable at $\phi_0$.
\end{proposition}
\begin{proof}
Suppose that $\phi_0\in\cV_k$, then we know that $h_{\infty,\phi_0}>0$.  However, since $\cO_k$ is dense in $LC_k(X,\bR)\cap\overline{B}(0,1)$, in any neighborhood of $\phi_0$, there is a potential in $\cO_k$ whose zero-temperature measure has entropy zero.  Therefore, the entropy map is not continuous, and, hence, not computable at $\phi_0$.

Suppose that $\phi_0\in\cV_k$.  In this case, since $\cO_k$ is dense in $LC_k(X,\bR)\cap\overline{B}(0,1)$, there is an infinite sequence of $\phi_n$'s in $\cO_k$ whose limit is $\phi$.  Since there are only finitely many $k$-elementary periodic points, by passing to a subsequence, we can assume that there is a $k$-elementary periodic point $x$ so that $\mu_x=\mu_{\infty,\phi_n}$ for all $n$.  If the zero-temperature measure map were continuous, then $\mu_{\infty,\phi_0}$ would be $\mu_x$, but this is not possible since the entropy of a periodic point measure is $0$.  Thus, the map $\phi\mapsto \mu_{\infty,\phi}$ is not continuous at $\phi_0\in\cV_k$.

Finally, to show that the map $\phi\mapsto \mu_{\infty,\phi}$ is not continuous at $\phi_0\in\cU_k$, we find two sequences of potentials converging to $\phi$ where the corresponding sequences of zero-temperature measures have different limits.  In particular, we construct two sequences of potentials $(\phi_{1,n})_n$ and $(\phi_{2,n})_n$ where, for all $n$, $\mu_{\infty,\phi_{1,n}}=\mu_x$ and $\mu_{\infty,\phi_{2,n}}=\mu_y$ with $x$ and $y$ distinct $k$-elementary periodic points.

Suppose, first, that $\phi_0\in\cU_k$.  Then, there are two $k$-elementary periodic points $x,y\in\EPer_{\max}^k(\phi)$ with disjoint $k$-cylinder support.  Therefore, there exists a $k$-cylinder $\cC(\tau)$ that is in the support of $x$, but not in the support of $y$.  Similarly, there is a $k$-cylinder $\cC(\tau')$ that is not in the support of $x$, but is in the support of $y$.  By (slightly) increasing $\phi_0|_{\cC(\tau)}$ or $\phi_0|_{\cC(\tau')}$, we can make $\EPer_{\max}^k(\phi)$ consist of a single $k$-elementary periodic point $x$ or $y$.  Therefore, by taking a sequence of small perturbations, we conclude that the function $\phi\mapsto \mu_{\infty,\phi}$ is not continuous at $\phi_0$.
\end{proof}

\section{Computability of Zero Temperature Measures for locally constant potentials: the case of unbounded cylinder length}\label{sec:5}

It is natural to ask whether Proposition \ref{prop:recursivelyopen} requires $k$ to be given or if the statements can be generalized to the sets $\cO=\bigcup_k \cO_k$ and $\cO\,\cup\,\cU$, where $\cU=\bigcup_k \cU_k$.  In this section, we give a negative answer to this question by showing the fact that the sets $\cO_k$ and $\cO_k\,\dot\cup\,\cU_k$ are recursively open does not extend to $\cO$ and $\cO\,\cup\,\cU$.  In particular, we prove  that $\cO$ has no interior points in $LC(X,\bR)$ (Theorem C).
We begin with an illustrative example that shows that $\cO$ is not open, in general, and provides the motivation for the proof of Theorem C.

\begin{example}\label{ex4.1}
Consider the SFT with alphabet $\{0,1,2,3\}$ and transition matrix given in Figure \ref{fig:states}.
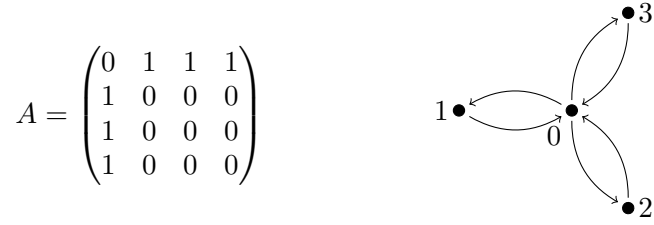
\begin{figure}[hbt]
\begin{tabular}{m{2in}m{2in}}
$
\displaystyle A=\begin{pmatrix}0&1&1&1\\1&0&0&0\\1&0&0&0\\1&0&0&0\end{pmatrix}
$
&
\begin{tikzpicture}[scale=.75]
\filldraw (0,0) circle (.1);
\node at (0,0) (A) {};
\filldraw (-2,0) circle (.1);
\node at (-2,0) (B) {};
\begin{scope}[rotate=120]
\filldraw (-2,0) circle (.1);
\node at (-2,0) (C) {};
\end{scope}
\begin{scope}[rotate=-120]
\filldraw (-2,0) circle (.1);
\node at (-2,0) (D) {};
\end{scope}
\draw (A) edge[->,in=30,out=150] (B);
\draw (A) edge[<-,in=-30,out=-150] (B);
\draw (A) edge[->,in=150,out=270] (C);
\draw (A) edge[<-,in=90,out=330] (C);
\draw (A) edge[->,in=210,out=90] (D);
\draw (A) edge[<-,in=270,out=30] (D);
\node[below left] at (0,-.1) {$0$};
\node[left] at (B) {$1$};
\node[right] at (C) {$2$};
\node[right] at (D) {$3$};
\end{tikzpicture}
\end{tabular}
\caption{The transition matrix and corresponding directed graph illustrating the allowable transitions between states.\label{fig:states}}
\end{figure}

Let $\phi\in LC_2(X,\bR)$ be the potential whose value on cylinders $\cC_2(01)$ and $\cC_2(10)$ is $2$, while its value on any other (nonempty) cylinder of length 2 is $1$.  In other words, $\phi$ is defined by the following matrix:
$$
\begin{pmatrix}
0&2&1&1\\
2&0&0&0\\
1&0&0&0\\
1&0&0&0
\end{pmatrix}
$$
For each $n\in\bN$, we define a potential $\phi_n\in LC_{2n+2}(X,\bR)$ which is a perturbation of $\phi$.  For a segment $\tau$ we denote by $\#_2(\tau)$ and $\#_3(\tau)$  the number of $2$'s or $3$'s appearing in $\tau$, respectively.  We define
$$
\phi_n(w)=\begin{cases}2+\frac{2}{n}&w\in \cC_2(01)\cup\cC_2(10),\,  \#_2(\pi_{2n+2}(w))+\#_3(\pi_{2n+2}(w))=1\\
\phi(w)&\text{otherwise}
\end{cases}.
$$
In other words, $\phi_n(w)=\phi(w)$ unless $\tau=\pi_{2n+2}(w)$ begins with $01$ or $10$ and contains either (exactly one $2$ and no $3$'s) or (exactly one $3$ and no $2$'s).  We see that $\|\phi-\phi_n\|_\infty=\frac{2}{n}$.  Moreover, $\EPer_{\max}^2(\phi)=\EPer_{\max}^{2n+2}(\phi)$ consists of the single $2$-elementary periodic orbit of $x=\cO(01)$.  We observe that both $\phi$ and $\phi_n$ are constant on the orbit of $x$, so $\mu_x(\phi)=2=\mu_x(\phi_n)$.

On the other hand, $\EPer_{\max}^{2n+2}(\phi_n)$ contains at least three $(2n+2)$-elementary periodic orbits: the orbits generated by $(01)^n02$, $(01)^n03$, and $(01)^n02(01)^n03$.  Here, $(01)^n$ represents the sequence of length $2n$ consisting of $01$ repeated $n$ times.  Let $z_1=\cO((01)^n02)$, $z_2=\cO((01)^n03)$, and $z_3=\cO((01)^n02(01)^n03)$.  We observe that $\mu_{z_i}(\phi_n)=2+\frac{1}{n+1}>2$ for $i=1,2,3$.  On the other hand, we observe that $\mu_{z_i}(\phi)=2-\frac{1}{n+1}$.  

Putting this together, we note that since $\EPer_{\max}^2(\phi)$ consists of a single periodic orbit, $\phi\in\cO_2$ with $h_{\infty,\phi}=0$ and $\mu_{\infty,\phi}=\mu_x$.  On the other hand, since $z_1$, $z_2$, and $z_3$ have overlapping cylinders, $\phi_n\in\cV_{2n+2}$ with $h_{\infty,\phi_n}>0$.  We, therefore, conclude that since $\phi_n\rightarrow\phi$, $\cO$ is not open in the supremum norm topology on $LC(X,\bR)$, so, in particular, $\cO$ is not a recursively open set.  We observe, however, that by Lemma \ref{lem2a}, $h_{\infty,\phi_n}\rightarrow 0$ as $n\to\infty$.
\end{example}
This example shows that, in general, $\cO$ is not open in $LC(X,\bR)$.  Moreover, we note that, in our example, the maximal $(2n+2)$-elementary periodic orbits of $\phi_n$ do not include the maximal $2$-elementary periodic orbits of $\phi$.  In other words, the set of maximizing elementary periodic orbits may change considerably under perturbations once the cylinder length is not fixed.  

Using this example as a guide, we show that the set $\cU\,\dot\cup\,\cV$ is dense in $LC(X,\bR)$, where $\cV=\bigcup_k\cV_k$.  This shows that the proof of Proposition \ref{prop:recursivelyopen} does not directly extend to the sets $\cO$ and $\cO\cup\cU$.  .

\begin{proposition}\label{prop:UV}
Let $f:X\to X$ be transitive SFT with positive topological entropy.  Then, the set $\cU\,\dot\cup\,\cV$ is dense in $LC(X,\bR)$ with respect to the supremum norm topology. 
\end{proposition}
\begin{proof}
Let $\phi\in\cO$. We show that for every neighborhood of $\phi$ there exists $\phi'\in\cU\,\dot\cup\,\cV$ in this neighborhood.  Since $\cO$ is dense, the density of $\cU\,\dot\cup\,\cV$ follows.  

Let $k\in \bN$ be such that $\phi\in\cO_k$, and let $x\in\EPer_{\max}^k(\phi)$ correspond to the unique maximal $k$-elementary periodic orbit for $\phi$ with period $\ell_x$ and generating segment $\tau_x$.  We now consider periodic points of the form $z_m=\cO(\tau_x^my)$, where $\tau_x^m$ denotes the $m$-times concatenation of $\tau_x$ and $y$ is a segment of length $\ell_y$.  By  transitivity and positive topological entropy of $f$ we may assume  that  $y_i\not=x_i$ for some
$i\in \{1,\dots ,\min\{\ell_x,\ell_y\}\}$.

In the following, we fix the segment $y$  and vary $m\geq 2$.
Let $\ell=\ell(m)$ be the smallest cylinder length so that $z_m$ is $\ell$-elementary periodic.  We observe that $(m-1)\ell_x<\ell$ since the $(m-1)\ell_x$-cylinders starting at the first two copies of $\tau_x$ in $z_m$ are identical.  On the other hand, $\ell\leq m\ell_x+\ell_y$ is a consequence of the construction of $z_m$.  We restrict our attention to cylinders of length $\ell$ throughout the remainder of this proof.  The fact that $x$ is a $k$-elementary periodic point with period $\ell_x$ implies that $|\cS_\ell(x)|=\ell_x$, where $\cS_\ell(x)$ denotes the $\ell$-cylinder support of $x$, see Equation \eqref{defkcyl}.  Moreover, since $z_m$ is $\ell$-elementary periodic with period $m\ell_x+\ell_y$, $|\cS_\ell(z_m)|=m\ell_x+\ell_y$.  We define potentials $\phi_{\epsilon,\ell}$ as follows:
$$
\phi_{\epsilon,\ell}(w)=\begin{cases}
\phi(w)+\epsilon&\cC_\ell(w)\in\cS_\ell(z_m)\setminus\cS_\ell(x)\\
\phi(w)&\text{otherwise}
\end{cases}.
$$
We observe that $\left\|\phi-\phi_{\epsilon,\ell}\right\|_\infty=\epsilon$.  Moreover, by construction, $\mu_x\left(\phi_{\epsilon,\ell}\right)=\mu_x(\phi)$.  On the other hand, since $|\cS_\ell(z_m)\setminus\cS_\ell(x)|\geq (m-1)\ell_x+\ell_y$, it follows that 
$$
\mu_{z_m}(\phi)+\frac{(m-1)\ell_x+\ell_y}{m\ell_x+\ell_y}\epsilon\leq\mu_{z_m}\left(\phi_{\epsilon,\ell}\right)\leq\mu_{z_m}(\phi)+\epsilon.
$$
Furthermore, since $z_m$ begins with $m$ copies of $\tau_x$, for $0\leq i<m\ell_x-k$, $\cC_k(f^i(z_m))\in\cS_k(x)$.  Let $m'=\left\lfloor m-\frac{\ell_y+k}{\ell_x}\right\rfloor$.  Then, $z_m=\cO(\tau_x^{m'}\tau_x^{m-m'}y)$ and
$$
\mu_{z_m}(\phi)=\frac{m'\ell_x}{m\ell_x+\ell_y}\mu_x(\phi)+\frac{(m-m')\ell_x+\ell_y}{m\ell_x+\ell_y}\mu_{\cO(\tau_x^{m-m'}y)}(\phi).
$$
For fixed $\epsilon>0$, we observe that as $m$ (and hence $m'$) increases, 
$$
\mu_{z_m}(\phi_{\epsilon,\ell})\rightarrow\mu_{z_m}(\phi)+\epsilon
\quad\text{and}\quad 
\mu_{z_m}(\phi)\rightarrow \mu_x(\phi).
$$
Therefore, for any fixed $\epsilon>0$, there exists an $m$ sufficiently large so that $\mu_{z_m}(\phi_{\epsilon,\ell})>\mu_{x}(\phi_{\epsilon,\ell})=\mu_x(\phi)$.  For the remainder of the proof, fix such $\epsilon$, $m$ and $\ell$.

Finally, we consider the family of potentials $\phi_{t,\ell}$, where $0\leq t\leq \epsilon$.  Let $t_0=\sup\left\{t:x\in\EPer_{\max}^\ell\left(\phi_{t,\ell}\right)\right\}$.  We observe that $t_0>0$ since $x\in\EPer_{\max}^\ell\left(\phi_{0,\ell}\right)$ and $\cO_\ell$ is open.  On the other hand, $t_0<\epsilon$ since $x\not\in\EPer_{\max}^\ell\left(\phi_{\epsilon,\ell}\right)$.  At $t_0$, $\EPer_{\max}^\ell\left(\phi_{t_0,\ell}\right)$ must contain at least two elementary periodic orbits, $x$ and some other orbit.  Therefore, $\phi_{t_0,\ell}\not\in\cO$ and $\left\|\phi-\phi_{t_0,\ell}\right\|_\infty=t<\epsilon$.  Therefore, $\phi_{t_0,\ell}\in\cU\,\dot\cup\,\cV$, and, by allowing $\epsilon$ decreasing to zero, the conclusion follows.
\end{proof}

\bibliographystyle{plain}
\bibliography{BW_References}

\end{document}